\documentclass{article}

\usepackage{amsmath}
\usepackage{amsthm}
\usepackage{alltt}
\usepackage{hyperref}

\newtheorem{theorem}{Theorem}[section]

\numberwithin{equation}{section}

\newcommand{\res}[1]
{\textrm{\rm Res}_{#1}}

\newcommand{\multsum}
{ \mathop{\sum\cdots\sum}_{
   \genfrac{}{}{0pt}{}{\scriptstyle k_1=0\ \cdots\ k_m=0}
     {\scriptstyle \sum_{i=1}^m k_i = n}
   }^{a_1\ \ \cdots\ \ a_m}
}
\newcommand{\multsumder}
{ \mathop{\sum\cdots\sum}_{
   \scriptstyle k_1=0\cdots k_{m-1}=0
   }^{a_1\ \cdots\ a_{m-1}}
}
\newcommand{\multsuminf}
{ \mathop{\sum\cdots\sum}_{
   \scriptstyle k_1=0\cdots k_{m-1}=0
   }^{\infty\ \ \cdots\ \ \infty}
}
\newcommand{\multsumij}
{ \mathop{\sum\cdots\sum}_{
   \scriptstyle i_1=0\ \cdots\ i_s=0}^{m\ \ \cdots\ \ m}
}
\newcommand{\multsumfree}
{ \mathop{\sum\cdots\sum}_{
   \scriptstyle k_1=0\ \cdots\ k_m=0 
    }^{a_1\ \ \cdots\ \ a_m}
}
\newcommand{\dblsum}[2]
{ \mathop{\sum\sum}_{
   \genfrac{}{}{0pt}{}{\scriptstyle #1=1\ #2=1}
     {\scriptstyle #1\neq #2}}^{m\ \ \ m}
}
\newcommand{\tripsum}[3]
{ \mathop{\sum\sum\sum}_{
   \genfrac{}{}{0pt}{}{\scriptstyle #1=1\ #2=1\ #3=1}
     {\scriptstyle #1\neq #2\neq #3,\ #1\neq #3}}^{m\ \ \ m\ \ \ m}
}
\newcommand{\mulprod}[3]
{ \prod_{\genfrac{}{}{0pt}{}{\scriptstyle #1}{\scriptstyle #2}}^{#3}
}
\newcommand{\re}
{\textrm{\rm Re}}
\newcommand{\im}
{\textrm{\rm Im}}
\newcommand{\Aabs}
{A_{\textrm{\rm abs}}}
\newcommand{\Cabs}
{C_{\textrm{\rm abs}}}

\title{Some Generalized Multi-sum Chu-Vandermonde Identities}
\author{M.J. Kronenburg}
\date{}

\begin{document}

\maketitle

\begin{abstract}
Some generalized multi-sum Chu-Vandermonde identities are presented and proved,
generalizing some known multi-sum Chu-Vandermonde identities from literature
and adding some quadratic and cubic examples of these identities.
Some other closely related identities are provided.
The identities are proved with the integral representation method
using complex residues, a method also used in some earlier papers.
\end{abstract}


\section{Generalized Multi-sum Chu-Vandermonde\\ Identities}

Let $(a_i)_{i=1}^m$ and $(c_i)_{i=1}^m$ be two sequences each of $m$ nonnegative integers,
and let $(x_i)_{i=1}^m$ and $(y_i)_{i=1}^m$ be two sequences each of $m$ complex numbers,
and let $\overline{w}$ be the complex conjugate of $w$.
Let the following definitions be given.
\begin{equation}
 A_{p,q} = \sum_{i=1}^m x_i^p a_i^q
\end{equation}
\begin{equation}
 A_p = A_{p,1} = \sum_{i=1}^m x_i^p a_i
\end{equation}
\begin{equation}
 A_{p,q}^* = \sum_{i=1}^m x_i^p y_i^q a_i
\end{equation}
\begin{equation}
 A_p^* = A_{0,p}^* = \sum_{i=1}^m y_i^p a_i
\end{equation}
\begin{equation}
 \Aabs = \sum_{i=1}^m |x_i|^2 a_i
\end{equation}
\begin{equation}
 C_{p,q} = \sum_{i=1}^m x_i^p c_i^q
\end{equation}
\begin{equation}
 C_p = C_{p,1} = \sum_{i=1}^m x_i^p c_i
\end{equation}
\begin{equation}
 C_{p,q}^* = \sum_{i=1}^m x_i^p y_i^q c_i
\end{equation}
\begin{equation}
 C_p^* = C_{0,p}^* = \sum_{i=1}^m y_i^p c_i
\end{equation}
\begin{equation}
 \Cabs = \sum_{i=1}^m |x_i|^2 c_i
\end{equation}
\begin{equation}
 S_{p,q} = \sum_{i=1}^m x_i a_i^p c_i^q
\end{equation}
The following generalized multi-sum Chu-Vandermonde identities are proved.\\
For integer $m\geq 1$ and $n\geq 0$:
\begin{equation}\label{res1}
 \multsum\prod_{i=1}^m \binom{a_i}{k_i}\binom{k_i}{c_i} 
  = \binom{A_0-C_0}{n-C_0} \prod_{i=1}^m \binom{a_i}{c_i}
\end{equation}
\begin{equation}\label{res2}
\begin{split}
 & \multsum \left[ \prod_{i=1}^m \binom{a_i}{k_i}\binom{k_i}{c_i} \right] \sum_{i=1}^m x_i k_i \\
 & =  \binom{A_0-C_0}{n-C_0} \left[ \prod_{i=1}^m \binom{a_i}{c_i} \right]
  \frac{(n-C_0)A_1+(A_0-n)C_1}{A_0-C_0} \\
\end{split}
\end{equation}
\begin{equation}\label{res3}
\begin{split}
 & \multsum \left[ \prod_{i=1}^m \binom{a_i}{k_i}\binom{k_i}{c_i} \right] (\sum_{i=1}^m x_i k_i)^2
  =  \binom{A_0-C_0}{n-C_0} \left[ \prod_{i=1}^m \binom{a_i}{c_i} \right] \\
 & \cdot \frac{(n-C_0)[(n-C_0-1)A_1^2+(A_0-n)(A_2-C_2+2A_1C_1)]+(A_0-n)(A_0-n-1)C_1^2}{(A_0-C_0)(A_0-C_0-1)} \\
\end{split}
\end{equation}
\begin{equation}\label{res4}
\begin{split}
 & \multsum \left[ \prod_{i=1}^m \binom{a_i}{k_i}\binom{k_i}{c_i} \right] |\sum_{i=1}^m x_i k_i|^2 \\
 & =  \binom{A_0-C_0}{n-C_0} \left[ \prod_{i=1}^m \binom{a_i}{c_i} \right] \frac{1}{(A_0-C_0)(A_0-C_0-1)} \\
 & \qquad\cdot \{ (n-C_0)[(n-C_0-1)|A_1|^2+(A_0-n)(\Aabs-\Cabs+A_1\overline{C_1}+\overline{A_1}C_1)] \\
 & \qquad\quad +(A_0-n)(A_0-n-1)|C_1|^2 \} \\
\end{split}
\end{equation}
\begin{equation}\label{res5}
\begin{split}
 & \multsum \left[ \prod_{i=1}^m \binom{a_i}{k_i}\binom{k_i}{c_i} \right] (\sum_{i=1}^m x_i k_i)(\sum_{i=1}^m y_i k_i) \\
 & =  \binom{A_0-C_0}{n-C_0} \left[ \prod_{i=1}^m \binom{a_i}{c_i} \right] \frac{1}{(A_0-C_0)(A_0-C_0-1)} \\
 & \quad\cdot \{ (n-C_0)[(n-C_0-1)A_1A_1^*+(A_0-n)(A_{1,1}^*-C_{1,1}^*+A_1C_1^*+A_1^*C_1)] \\
 & \quad\quad +(A_0-n)(A_0-n-1)C_1C_1^* \} \\
\end{split}
\end{equation}
\begin{equation}\label{res6}
\begin{split}
 & \multsum \left[ \prod_{i=1}^m \binom{a_i}{k_i}\binom{k_i}{c_i} \right] \sum_{i=1}^m x_i k_i^2 
  =  \binom{A_0-C_0}{n-C_0} \left[ \prod_{i=1}^m \binom{a_i}{c_i} \right] \\
 & \cdot \frac{ (n-C_0)[(n-C_0-1)A_{1,2}+(A_0-n)(A_1-C_1+2S_{1,1})] +(A_0-n)(A_0-n-1)C_{1,2} }{(A_0-C_0)(A_0-C_0-1)} \\
\end{split}
\end{equation}
\begin{equation}\label{res7}
\begin{split}
 & \multsum \left[ \prod_{i=1}^m \binom{a_i}{k_i}\binom{k_i}{c_i} \right] (\sum_{i=1}^m x_i k_i)^3 \\
 & =  \binom{A_0-C_0}{n-C_0} \left[ \prod_{i=1}^m \binom{a_i}{c_i} \right] \frac{1}{(A_0-C_0)(A_0-C_0-1)(A_0-C_0-2)} \\
 & \qquad\cdot \{ (n-C_0)(n-C_0-1) [ (n-C_0-2)A_1^3 \\ 
 & \qquad\qquad\qquad +(A_0-n)(C_3-A_3+3A_1(A_2-C_2+A_1C_1))] \\
 & \qquad\quad +(A_0-n)(A_0-n-1) [ (A_0-n-2)C_1^3 \\
 & \qquad\qquad\qquad + (n-C_0)(A_3-C_3+3C_1(A_2-C_2+A_1C_1))] \} \\
\end{split}
\end{equation}
\begin{equation}\label{res8}
\begin{split}
 & \multsum \left[ \prod_{i=1}^m \binom{a_i}{k_i}\binom{k_i}{c_i} \right] \sum_{i=1}^m x_i k_i^3 \\
 & =  \binom{A_0-C_0}{n-C_0} \left[ \prod_{i=1}^m \binom{a_i}{c_i} \right] \frac{1}{(A_0-C_0)(A_0-C_0-1)(A_0-C_0-2)} \\
 & \qquad\cdot \{ (n-C_0)(n-C_0-1) [ (n-C_0-2)A_{1,3} \\ 
 & \qquad\qquad\qquad +(A_0-n)(C_1-A_1+3(S_{2,1}+A_{1,2}-S_{1,1}))] \\
 & \qquad\quad +(A_0-n)(A_0-n-1) [ (A_0-n-2)C_{1,3} \\
 & \qquad\qquad\qquad + (n-C_0)(A_1-C_1+3(S_{1,2}-C_{1,2}+S_{1,1}))] \} \\
\end{split}
\end{equation}
Taking all $c_i=0$ and therefore all $C_{p,q}=C_p=C_{p,q}^*=C_p^*=\Cabs=0$, identities (\ref{res1}), (\ref{res2}) and (\ref{res4})
can be found in \cite{M18},
and taking $m=2$, $x_1=1$ and $x_2=0$, and therefore $A_0=a+b$, $C_0=c+d$, $A_1=A_2=A_3=a$ and $C_1=C_2=C_3=c$,
identities (\ref{res1}), (\ref{res2}), (\ref{res3}) and (\ref{res7})
can be found in \cite{K18}.\\

\section{Unrestricted Multi-sum Identities}

When the multiple sum is without the restriction on the $k_i$-indices, other closely related identities result which are given below.\\
For integer $m\geq 1$:
\begin{equation}
 \multsumfree \prod_{i=1}^m \binom{a_i}{k_i}\binom{k_i}{c_i} = 2^{A_0-C_0} \prod_{i=1}^m \binom{a_i}{c_i}
\end{equation}
\begin{equation}
 \multsumfree \left[ \prod_{i=1}^m \binom{a_i}{k_i}\binom{k_i}{c_i} \right] \sum_{i=1}^m x_i k_i
  = 2^{A_0-C_0-1} \left[ \prod_{i=1}^m \binom{a_i}{c_i} \right] (A_1+C_1)
\end{equation}
\begin{equation}
\begin{split}
 & \multsumfree \left[ \prod_{i=1}^m \binom{a_i}{k_i}\binom{k_i}{c_i} \right] (\sum_{i=1}^m x_i k_i)^2 \\
 & = 2^{A_0-C_0-2} \left[ \prod_{i=1}^m \binom{a_i}{c_i} \right] [ A_2-C_2 + (A_1+C_1)^2 ] \\
\end{split}
\end{equation}
\begin{equation}
\begin{split}
 & \multsumfree \left[ \prod_{i=1}^m \binom{a_i}{k_i}\binom{k_i}{c_i} \right] |\sum_{i=1}^m x_i k_i|^2 \\
 & = 2^{A_0-C_0-2} \left[ \prod_{i=1}^m \binom{a_i}{c_i} \right] ( \Aabs-\Cabs + |A_1+C_1|^2 ) \\
\end{split}
\end{equation}
\begin{equation}
\begin{split}
 & \multsumfree \left[ \prod_{i=1}^m \binom{a_i}{k_i}\binom{k_i}{c_i} \right] (\sum_{i=1}^m x_i k_i)(\sum_{i=1}^m y_i k_i) \\
 & = 2^{A_0-C_0-2} \left[ \prod_{i=1}^m \binom{a_i}{c_i} \right] [ A_{1,1}^*-C_{1,1}^* + (A_1+C_1)(A_1^*+C_1^*) ] \\
\end{split}
\end{equation}
\begin{equation}
\begin{split}
 & \multsumfree \left[ \prod_{i=1}^m \binom{a_i}{k_i}\binom{k_i}{c_i} \right] \sum_{i=1}^m x_i k_i^2 \\
 & = 2^{A_0-C_0-2} \left[ \prod_{i=1}^m \binom{a_i}{c_i} \right] ( A_1-C_1+A_{1,2}+C_{1,2}+2S_{1,1} ) \\
\end{split}
\end{equation}
\begin{equation}
\begin{split}
 & \multsumfree \left[ \prod_{i=1}^m \binom{a_i}{k_i}\binom{k_i}{c_i} \right] (\sum_{i=1}^m x_i k_i)^3 \\
 & = 2^{A_0-C_0-3} \left[ \prod_{i=1}^m \binom{a_i}{c_i} \right] (A_1+C_1)[(A_1+C_1)^2 + 3(A_2-C_2)] \\
\end{split}
\end{equation}
\begin{equation}
\begin{split}
 & \multsumfree \left[ \prod_{i=1}^m \binom{a_i}{k_i}\binom{k_i}{c_i} \right] \sum_{i=1}^m x_i k_i^3 \\
 & = 2^{A_0-C_0-3} \left[ \prod_{i=1}^m \binom{a_i}{c_i} \right] [ A_{1,3}+C_{1,3}+3(A_{1,2}-C_{1,2}+S_{1,2}+S_{2,1}) ] \\
\end{split}
\end{equation}

\section{Proof of the Multi-sum Chu-Vandermonde\\ Identities}

For the proof of the multi-sum identities, the order of summation is changed:
\begin{equation}\label{multsumorder}
\begin{split}
 & \multsum\left[ \prod_{i=1}^m \binom{a_i}{k_i}\binom{k_i}{c_i} \right]\multsumij\prod_{j=1}^s w_{i_j} k_{i_j}^{p_j} \\
 & = \multsumij \left[\prod_{j=1}^s w_{i_j}\right] \multsum\left[ \prod_{i=1}^m \binom{a_i}{k_i}\binom{k_i}{c_i} \right] \prod_{j=1}^s k_{i_j}^{p_j} \\
\end{split}
\end{equation}
The powers $p_j$ in the right side must be constant in the second multiple summation over the $k_i$, 
which means that the indices $i_j$ must be mutually unequal.
A product of sums in the summand is therefore splitted into multiple
sums over mutually unequal indices:
\begin{equation}
 (\sum_{i=1}^m x_i k_i)(\sum_{i=1}^m y_i k_i) = 
  \dblsum{p}{q} x_p y_q k_p k_q 
  + \sum_{p=1}^m x_p y_p k_p^2
\end{equation}
\begin{equation}\label{tripexpand}
\begin{split}
 & (\sum_{i=1}^m x_i k_i)(\sum_{i=1}^m y_i k_i)(\sum_{i=1}^m z_i k_i) 
 = \tripsum{p}{q}{r} x_py_qz_r k_pk_qk_r \\
 & + \dblsum{p}{q} ( x_py_qz_q + y_px_qz_q + z_px_qy_q ) k_p k_q^2 + \sum_{p=1}^m x_p y_p z_p k_p^3 \\
\end{split}
\end{equation}
For the multi-sum identities up to cubic order, in the right side of (\ref{multsumorder}) the
power combinations $1$, $k_p$, $k_pk_q$, $k_p^2$, $k_pk_qk_r$, $k_pk_q^2$ and $k_p^3$ are needed,
where as mentioned $p\neq q\neq r$ and $p\neq r$,
and the second multiple summation over the $k_i$ is evaluated below
with the integral representation method using complex residues as in earlier papers \cite{K18,K19}:
\begin{equation}\label{binomdef}
 \binom{n}{k} = \res{w} \frac{(1+w)^n}{w^{k+1}}
\end{equation}
When $f(w)$ does not have a pole at $w=w_p$, then 
the following is proved in section \ref{ressect}:
\begin{equation}\label{resdef}
 \res{w=w_p} \frac{f(w)}{(w-w_p)^k} = \frac{1}{(k-1)!} D_w^{k-1} f(w)|_{w=w_p}
\end{equation}
where $D_w^n f(w)|_{w=w_p}$ is the $n$-th derivative of $f(w)$ at $w=w_p$.
\begin{theorem}\label{theorem0}
For integer $m\geq 1$:
\begin{equation}
 \multsum\prod_{i=1}^m \binom{a_i}{k_i}\binom{k_i}{c_i} 
  = \binom{A_0-C_0}{n-C_0} \prod_{i=1}^m \binom{a_i}{c_i}
\end{equation}
\end{theorem}
\begin{proof}
The trinomial revision identity \cite{GKP94,K15} is applied:
\begin{equation}\label{trirev}
 \binom{a_i}{k_i}\binom{k_i}{c_i} = \binom{a_i-c_i}{k_i-c_i}\binom{a_i}{c_i}
\end{equation}
Because this expression is zero when $c_i>a_i$, 
in the multiple summation it can be assumed that $0\leq c_i\leq a_i$
and that $c_i\leq k_i\leq a_i$.
The restriction on the $k_i$-indices in the multiple sum can therefore be realized by
setting $k_m$ equal to:
\begin{equation}\label{ksumdef}
 k_m = n - \sum_{i=1}^{m-1} k_i
\end{equation}
reducing an $m$-fold sum to an $(m-1)$-fold sum.
The infinite geometric series \cite{K19} is used:
\begin{equation}\label{infgeom0}
 \sum_{k=0}^{\infty} w^k = \frac{1}{1-w}
\end{equation}
The upper bounds $a_i$ can be replaced by $\infty$ because
the summand is zero when any $k_i>a_i$.
\begin{equation}
\begin{split}
 & \multsum\prod_{i=1}^m \binom{a_i}{k_i}\binom{k_i}{c_i} \\
 & = \left[ \prod_{i=1}^m \binom{a_i}{c_i} \right] \multsumder
 \binom{a_m-c_m}{n-c_m-\sum_{i=1}^{m-1}k_i} \prod_{i=1}^{m-1} \binom{a_i-c_i}{k_i-c_i} \\
 & = \left[ \prod_{i=1}^m \binom{a_i}{c_i} \right] \multsuminf
 \left[\prod_{i=1}^m \res{w_i}\right] \frac{(1+w_m)^{a_m-c_m}}{w_m^{n-c_m-\sum_{i=1}^{m-1}k_i+1}}
   \prod_{i=1}^{m-1} \frac{(1+w_i)^{a_i-c_i}}{w_i^{k_i-c_i+1}} \\
 & = \left[ \prod_{i=1}^m \binom{a_i}{c_i}\res{w_i} \right] \frac{1}{w_m^n}
   \left[\prod_{i=1}^{m}(1+w_i)^{a_i-c_i}w_i^{c_i-1}\right] \multsuminf \prod_{i=1}^{m-1}\left(\frac{w_m}{w_i}\right)^{k_i} \\
 & = \left[ \prod_{i=1}^m \binom{a_i}{c_i}\res{w_i} \right] \frac{1}{w_m^n}
   \left[\prod_{i=1}^{m}(1+w_i)^{a_i-c_i}w_i^{c_i-1}\right] \prod_{i=1}^{m-1} \sum_{k_i=0}^{\infty} \left(\frac{w_m}{w_i}\right)^{k_i} \\
 & = \left[ \prod_{i=1}^m \binom{a_i}{c_i}\res{w_i} \right] \frac{1}{w_m^n}
    \left[\prod_{i=1}^{m}(1+w_i)^{a_i-c_i}w_i^{c_i-1}\right] \prod_{i=1}^{m-1} \frac{1}{1-\frac{\textstyle w_m}{\textstyle w_i}} \\
 & = \left[ \prod_{i=1}^m \binom{a_i}{c_i} \right] \res{w_m} \frac{(1+w_m)^{a_m-c_m}}{w_m^{n-c_m+1}}
     \prod_{i=1}^{m-1}\res{w_i}\frac{(1+w_i)^{a_i-c_i}w_i^{c_i}}{w_i-w_m} \\
\end{split}
\end{equation}
The residue of $w_i$ is evaluated with (\ref{resdef}), using that $c_i\geq 0$ and $a_i-c_i\geq 0$:
\begin{equation}
\begin{split}
 & = \left[ \prod_{i=1}^m \binom{a_i}{c_i} \right] \res{w_m} \frac{(1+w_m)^{a_m-c_m}}{w_m^{n-c_m+1}}
   \prod_{i=1}^{m-1} (1+w_m)^{a_i-c_i} w_m^{c_i} \\
 & = \left[ \prod_{i=1}^m \binom{a_i}{c_i} \right] \res{w_m} \frac{(1+w_m)^{\sum_{i=1}^m(a_i-c_i)}}{w_m^{n-\sum_{i=1}^mc_i+1}}
   = \left[ \prod_{i=1}^m \binom{a_i}{c_i} \right] \binom{\sum_{i=1}^m(a_i-c_i)}{n-\sum_{i=1}^mc_i}
\end{split}
\end{equation}
\end{proof}
\begin{theorem}\label{theorem1}
For integer $m\geq 1$ and $1\leq p\leq m$:
\begin{equation}
 \multsum \left[\prod_{i=1}^m \binom{a_i}{k_i}\binom{k_i}{c_i}\right] k_p 
  = \binom{A_0-C_0}{n-C_0} \left[\prod_{i=1}^m \binom{a_i}{c_i}\right] \frac{(n-C_0)a_p+(A_0-n)c_p}{A_0-C_0} 
\end{equation}
\end{theorem}
\begin{proof}
The following modified infinite geometric series \cite{K19} is used:
\begin{equation}
 \sum_{k=0}^{\infty} k w^k = \frac{w}{(1-w)^2}
\end{equation}
The theorem is first proved for $1\leq p<m$:
\begin{equation}
\begin{split}
 &  \left[ \prod_{i=1}^m \res{w_i} \right] \frac{1}{w_m^n}
   \left[\prod_{i=1}^{m}(1+w_i)^{a_i-c_i}w_i^{c_i-1}\right] \multsuminf k_p \prod_{i=1}^{m-1}\left(\frac{w_m}{w_i}\right)^{k_i} \\
 & = \left[ \prod_{i=1}^m \res{w_i} \right] \frac{1}{w_m^n}
    \left[\prod_{i=1}^{m}(1+w_i)^{a_i-c_i}w_i^{c_i-1}\right] \frac{\frac{\textstyle w_m}{\textstyle w_p}}{(1-\frac{\textstyle w_m}{\textstyle w_p})^2}
    \mulprod{i=1}{i\neq p}{m-1} \frac{1}{1-\frac{\textstyle w_m}{\textstyle w_i}} \\
 & = \res{w_m} \frac{(1+w_m)^{a_m-c_m}}{w_m^{n-c_m}} \res{w_p} \frac{(1+w_p)^{a_p-c_p}w_p^{c_p}}{(w_p-w_m)^2}
     \mulprod{i=1}{i\neq p}{m-1}\res{w_i}\frac{(1+w_i)^{a_i-c_i}w_i^{c_i}}{w_i-w_m} \\
 & = \res{w_m} \frac{(1+w_m)^{a_m-c_m}}{w_m^{n-c_m}} \res{w_p} \frac{(1+w_p)^{a_p-c_p}w_p^{c_p}}{(w_p-w_m)^2}
     \mulprod{i=1}{i\neq p}{m-1} (1+w_m)^{a_i-c_i} w_m^{c_i} \\
\end{split}
\end{equation}
The residue of $w_p$ is evaluated with (\ref{resdef}) using the product
rule for the derivative:
\begin{equation}\label{resquad}
 \res{w_p} \frac{(1+w_p)^{a_p-c_p}w_p^{c_p}}{(w_p-w_m)^2} =
 (a_p-c_p)(1+w_m)^{a_p-c_p-1}w_m^{c_p} + c_p(1+w_m)^{a_p-c_p}w_m^{c_p-1}
\end{equation}
and evaluating the residue of $w_m$ and using the absorption identity \cite{GKP94,K15}:
\begin{equation}
\begin{split}
 & \res{w_m} [ (a_p-c_p) \frac{(1+w_m)^{\sum_{i=1}^m(a_i-c_i)-1}}{w_m^{n-\sum_{i=1}^mc_i}} 
  + c_p \frac{(1+w_m)^{\sum_{i=1}^m(a_i-c_i)}}{w_m^{n-\sum_{i=1}^mc_i+1}} ] \\
 & = \binom{A_0-C_0-1}{n-C_0-1} (a_p-c_p) + \binom{A_0-C_0}{n-C_0} c_p \\
 & = \binom{A_0-C_0}{n-C_0} [ \frac{(a_p-c_p)(n-C_0)}{A_0-C_0} + c_p ] \\
 & = \binom{A_0-C_0}{n-C_0} \frac{(n-C_0)a_p+(A_0-n)c_p}{A_0-C_0}
\end{split}
\end{equation}
The theorem is now proved for $1\leq p<m$,
but because the indices of the $k_i$ in the left side of the theorem
can be interchanged, the theorem is true for $1\leq p\leq m$.
\end{proof}
\begin{theorem}\label{theoremres1}
For integer $m\geq 1$:
\begin{equation}
\begin{split}
 & \multsum \left[ \prod_{i=1}^m \binom{a_i}{k_i}\binom{k_i}{c_i} \right] \sum_{i=1}^m x_i k_i \\
 & =  \binom{A_0-C_0}{n-C_0} \left[ \prod_{i=1}^m \binom{a_i}{c_i} \right]
  \frac{(n-C_0)A_1+(A_0-n)C_1}{A_0-C_0} \\
\end{split}
\end{equation}
\end{theorem}
\begin{proof}
Changing the order of summation as in (\ref{multsumorder}) and using theorem \ref{theorem1}:
\begin{equation}
\begin{split}
 & \binom{A_0-C_0}{n-C_0} \left[ \prod_{i=1}^m \binom{a_i}{c_i} \right] 
   \sum_{i=1}^m x_i \frac{(n-C_0)a_i+(A_0-n)c_i}{A_0-C_0} \\
 & = \binom{A_0-C_0}{n-C_0} \left[ \prod_{i=1}^m \binom{a_i}{c_i} \right]
  \frac{(n-C_0)A_1+(A_0-n)C_1}{A_0-C_0} \\
\end{split}
\end{equation}
\end{proof}
\begin{theorem}\label{theorem2}
For integer $m\geq 1$ and $1\leq p\leq m$:
\begin{equation}
\begin{split}
 & \multsum \left[\prod_{i=1}^m \binom{a_i}{k_i}\binom{k_i}{c_i}\right] k_p^2 
   = \binom{A_0-C_0}{n-C_0} \left[\prod_{i=1}^m \binom{a_i}{c_i}\right] \\
 & \cdot \frac{(n-C_0)[(n-C_0-1)a_p^2+(A_0-n)(a_p-c_p+2a_pc_p)]+(A_0-n)(A_0-n-1)c_p^2}{(A_0-C_0)(A_0-C_0-1)} \\
\end{split}
\end{equation}
\end{theorem}
\begin{proof}
The following modified infinite geometric series \cite{K19} is used:
\begin{equation}
 \sum_{k=0}^{\infty} k^2 w^k = \frac{w}{(1-w)^2} + \frac{2w^2}{(1-w)^3}
\end{equation}
The first term gives a result identical to theorem \ref{theorem1},
and for the second term:
\begin{equation}
\begin{split}
 &  \left[ \prod_{i=1}^m \res{w_i} \right] \frac{1}{w_m^n}
    \left[\prod_{i=1}^{m}(1+w_i)^{a_i-c_i}w_i^{c_i-1}\right] 
     \frac{2(\frac{\textstyle w_m}{\textstyle w_p})^2}{(1-\frac{\textstyle w_m}{\textstyle w_p})^3}
    \mulprod{i=1}{i\neq p}{m-1} \frac{1}{1-\frac{\textstyle w_m}{\textstyle w_i}} \\
 & = \res{w_m} \frac{(1+w_m)^{a_m-c_m}}{w_m^{n-c_m-1}} \res{w_p} \frac{2(1+w_p)^{a_p-c_p}w_p^{c_p}}{(w_p-w_m)^3}
     \mulprod{i=1}{i\neq p}{m-1} (1+w_m)^{a_i-c_i} w_m^{c_i} \\
\end{split}
\end{equation}
The residue of $w_p$ is evaluated with (\ref{resdef}) using the product
rule for the second derivative:
\begin{equation}
\begin{split}
 & \res{w_p} \frac{2(1+w_p)^{a_p-c_p}w_p^{c_p}}{(w_p-w_m)^3} =
  (a_p-c_p)(a_p-c_p-1)(1+w_m)^{a_p-c_p-2}w_m^{c_p} \\
 & \quad + 2(a_p-c_p)c_p(1+w_m)^{a_p-c_p-1}w_m^{c_p-1} + c_p(c_p-1)(1+w_m)^{a_p-c_p}w_m^{c_p-2} \\
\end{split}
\end{equation}
Using the absorption identity \cite{GKP94,K15}:
\begin{equation}
\begin{split}
 & \binom{A_0-C_0-2}{n-C_0-2}(a_p-c_p)(a_p-c_p-1) + 2 \binom{A_0-C_0-1}{n-C_0-1}(a_p-c_p)c_p \\
 & + \binom{A_0-C_0}{n-C_0}c_p(c_p-1) \\
 & = \binom{A_0-C_0}{n-C_0} [ \frac{(n-C_0)(n-C_0-1)}{(A_0-C_0)(A_0-C_0-1)}(a_p-c_p)(a_p-c_p-1) \\
 & \qquad\qquad + 2 \frac{n-C_0}{A_0-C_0}(a_p-c_p)c_p + c_p(c_p-1) ] \\
 & = \binom{A_0-C_0}{n-C_0} \frac{1}{(A_0-C_0)(A_0-C_0-1)} \\
 & \qquad \cdot \{(n-C_0)(a_p-c_p) [ (n-C_0-1)(a_p-c_p-1)+2(A_0-C_0-1)c_p ] \\
 & \qquad\quad + (A_0-C_0)(A_0-C_0-1)c_p(c_p-1) \} \\
\end{split}
\end{equation}
As mentioned the result of theorem \ref{theorem1} must be added, which gives:
\begin{equation}
\begin{split}
 & \binom{A_0-C_0}{n-C_0} \frac{1}{(A_0-C_0)(A_0-C_0-1)} \\
 & \qquad \cdot \{(n-C_0)(a_p-c_p) [ (n-C_0-1)(a_p-c_p-1)+2(A_0-C_0-1)c_p ] \\
 & \qquad\quad + (A_0-C_0-1)[ (A_0-C_0)c_p(c_p-1) + (n-C_0)a_p + (A_0-n)c_p ] \} \\
\end{split}
\end{equation}
This result simplifies to the theorem.
\end{proof}
\begin{theorem}\label{theorem11}
For integer $m\geq 2$, $1\leq p\leq m$, $1\leq q\leq m$ and $p\neq q$:
\begin{equation}
\begin{split}
 & \multsum \left[\prod_{i=1}^m \binom{a_i}{k_i}\binom{k_i}{c_i}\right] k_p k_q 
   = \binom{A_0-C_0}{n-C_0} \left[\prod_{i=1}^m \binom{a_i}{c_i}\right] \\
 & \cdot \frac{(n-C_0)[(n-C_0-1)a_pa_q+(A_0-n)(a_pc_q+c_pa_q)]+(A_0-n)(A_0-n-1)c_pc_q}{(A_0-C_0)(A_0-C_0-1)} \\
\end{split}
\end{equation}
\end{theorem}
\begin{proof}
The proof is similar to theorem \ref{theorem1} but with two quadratic residues
instead of one.
\begin{equation}
\begin{split}
 & \res{w_m} \frac{(1+w_m)^{a_m-c_m}}{w_m^{n-c_m-1}} \res{w_p} \frac{(1+w_p)^{a_p-c_p}w_p^{c_p}}{(w_p-w_m)^2}
    \res{w_q} \frac{(1+w_q)^{a_p-c_p}w_q^{c_p}}{(w_q-w_m)^2} \\
 & \qquad \cdot \mulprod{i=1}{i\neq p, i\neq q}{m-1} (1+w_m)^{a_i-c_i} w_m^{c_i} \\
\end{split}
\end{equation}
The two last residues are given by (\ref{resquad}), and multiplying these out gives:
\begin{equation}
\begin{split}
 & [(a_p-c_p)(1+w_m)^{a_p-c_p-1}w_m^{c_p} + c_p(1+w_m)^{a_p-c_p}w_m^{c_p-1} ] \\
 & \cdot [(a_q-c_q)(1+w_m)^{a_q-c_q-1}w_m^{c_q} + c_q(1+w_m)^{a_q-c_q}w_m^{c_q-1} ] \\
 & = (a_p-c_p)(a_q-c_q)(1+w_m)^{a_p-c_p+a_q-c_q-2}w_m^{c_p+c_q} \\
 & \quad + [ (a_p-c_p)c_q + (a_q-c_q)c_p ] (1+w_m)^{a_p-c_p+a_q-c_q-1} w_m^{c_p+c_q-1} \\
 & \quad + c_pc_q(1+w_m)^{a_p-c_p+a_q-c_q}w_m^{c_p+c_q-2} \\
\end{split}
\end{equation}
This yields the following:
\begin{equation}
\begin{split}
 & \binom{A_0-C_0-2}{n-C_0-2}(a_p-c_p)(a_q-c_q) + \binom{A_0-C_0-1}{n-C_0-1}  [ (a_p-c_p)c_q + (a_q-c_q)c_p ] \\
 & \quad + \binom{A_0-C_0}{n-C_0} c_pc_q \\
\end{split}
\end{equation}
and using again the absorption identity:
\begin{equation}
\begin{split}
 & \binom{A_0-C_0}{n-C_0} \frac{1}{(A_0-C_0)(A_0-C_0-1)} \\
 & \cdot \{ (n-C_0) [(n-C_0-1)(a_p-c_p)(a_q-c_q) + (A_0-C_0-1)((a_p-c_p)c_q+(a_q-c_q)c_p)] \\
 & \qquad + (A_0-C_0)(A_0-C_0-1)c_pc_q \} \\
\end{split}
\end{equation}
This result simplifies to the theorem.
\end{proof}
\begin{theorem}\label{theoremres2}
For integer $m\geq 1$:
\begin{equation}
\begin{split}
 & \multsum \left[ \prod_{i=1}^m \binom{a_i}{k_i}\binom{k_i}{c_i} \right] (\sum_{i=1}^m x_i k_i)(\sum_{i=1}^m y_i k_i) \\
 & =  \binom{A_0-C_0}{n-C_0} \left[ \prod_{i=1}^m \binom{a_i}{c_i} \right] \frac{1}{(A_0-C_0)(A_0-C_0-1)} \\
 & \quad\cdot \{ (n-C_0)[(n-C_0-1)A_1A_1^*+(A_0-n)(A_{1,1}^*-C_{1,1}^*+A_1C_1^*+A_1^*C_1)] \\
 & \quad\quad +(A_0-n)(A_0-n-1)C_1C_1^* \} \\
\end{split}
\end{equation}
\end{theorem}
\begin{proof}
The product in the summand becomes:
\begin{equation}
 (\sum_{i=1}^m x_i k_i)(\sum_{i=1}^m y_i k_i) 
  = \dblsum{i}{j} x_i y_j k_i k_j 
   + \sum_{i=1}^m x_iy_ik_i^2 
\end{equation}
Changing the order of summation as in (\ref{multsumorder}) and using theorem \ref{theorem2} and \ref{theorem11}.
\begin{equation}
\begin{split}
 & \dblsum{i}{j} x_iy_j \{ (n-C_0)[(n-C_0-1)a_ia_j + (A_0-n)(a_ic_j+c_ia_j)] \\
 & \qquad\qquad + (A_0-n)(A_0-n-1)c_ic_j \} \\
 & \qquad + \sum_{i=1}^m x_iy_i \{ (n-C_0)[(n-C_0-1)a_i^2 + (A_0-n)(a_i-c_i+2a_ic_i) ] \\
 & \qquad\qquad + (A_0-n)(A_0-n-1) c_i^2 \} \\
 & = \sum_{i=1}^m\sum_{j=1}^m x_iy_j \{ (n-C_0)[(n-C_0-1)a_ia_j + (A_0-n)(a_ic_j+c_ia_j)] \\
 & \qquad\qquad + (A_0-n)(A_0-n-1)c_ic_j \} \\
 & \qquad + \sum_{i=1}^m x_iy_i (n-C_0)(A_0-n) (a_i-c_i) \\
 & = (n-C_0) [(n-C_0-1)A_1A_1^* + (A_0-n)(A_{1,1}^*-C_{1,1}^*+A_1C_1^*+A_1^*C_1) ] \\
 & \qquad + (A_0-n)(A_0-n-1) C_1C_1^* \\
\end{split}
\end{equation}
\end{proof}
\begin{theorem}\label{theoremres3}
For integer $m\geq 1$:
\begin{equation}
\begin{split}
 & \multsum \left[ \prod_{i=1}^m \binom{a_i}{k_i}\binom{k_i}{c_i} \right] (\sum_{i=1}^m x_i k_i)^2
  =  \binom{A_0-C_0}{n-C_0} \left[ \prod_{i=1}^m \binom{a_i}{c_i} \right] \\
 & \cdot \frac{(n-C_0)[(n-C_0-1)A_1^2+(A_0-n)(A_2-C_2+2A_1C_1)]+(A_0-n)(A_0-n-1)C_1^2}{(A_0-C_0)(A_0-C_0-1)} \\
\end{split}
\end{equation}
\end{theorem}
\begin{proof}
In theorem \ref{theoremres2} taking $y_i=x_i$, then $A_1^*=A_1$, $C_1^*=C_1$, $A_{1,1}^*=A_2$ and $C_{1,1}^*=C_2$,
giving this theorem.
\end{proof}
\begin{theorem}\label{theoremabs}
\begin{equation}
\begin{split}
 & \multsum \left[ \prod_{i=1}^m \binom{a_i}{k_i}\binom{k_i}{c_i} \right] |\sum_{i=1}^m x_i k_i|^2 \\
 & =  \binom{A_0-C_0}{n-C_0} \left[ \prod_{i=1}^m \binom{a_i}{c_i} \right] \frac{1}{(A_0-C_0)(A_0-C_0-1)} \\
 & \qquad\cdot \{ (n-C_0)[(n-C_0-1)|A_1|^2+(A_0-n)(\Aabs-\Cabs+A_1\overline{C_1}+\overline{A_1}C_1)] \\
 & \qquad\quad +(A_0-n)(A_0-n-1)|C_1|^2 \} \\
\end{split}
\end{equation}
\end{theorem}
\begin{proof}
For this theorem the following is used.
\begin{equation}
  |\sum_{i=1}^m x_ik_i|^2 = \re^2(\sum_{i=1}^m x_ik_i) + \im^2(\sum_{i=1}^m x_ik_i) = (\sum_{i=0}^m \re(x_i)k_i)^2 + (\sum_{i=0}^m \im(x_i)k_i)^2 
\end{equation}
Both of these terms are evaluated by theorem \ref{theoremres3}, replacing $x_i$ by $\re(x_i)$ and then by $\im(x_i)$,
and then the results are added, using the following.
\begin{equation}
 (\sum_{i=1}^m \re(x_i)a_i)^2 + (\sum_{i=1}^m \im(x_i)a_i)^2 = \re^2(\sum_{i=1}^m x_ia_i) + \im^2(\sum_{i=1}^m x_ia_i)
 = |A_1|^2
\end{equation}
\begin{equation}
 (\sum_{i=1}^m \re(x_i)c_i)^2 + (\sum_{i=1}^m \im(x_i)c_i)^2 = \re^2(\sum_{i=1}^m x_ic_i) + \im^2(\sum_{i=1}^m x_ic_i)
 = |C_1|^2
\end{equation}
\begin{equation}
 \sum_{i=1}^m \re^2(x_i)a_i + \sum_{i=1}^m \im^2(x_i)a_i = \sum_{i=1}^m [\re^2(x_i)+\im^2(x_i)]a_i
 = \sum_{i=1}^m |x_i|^2 a_i = \Aabs
\end{equation}
\begin{equation}
 \sum_{i=1}^m \re^2(x_i)c_i + \sum_{i=1}^m \im^2(x_i)c_i = \sum_{i=1}^m [\re^2(x_i)+\im^2(x_i)]c_i
 = \sum_{i=1}^m |x_i|^2 c_i = \Cabs
\end{equation}
and using for complex $a$, $c$: $2[\re(a)\re(c)+\im(a)\im(c)]=a\overline{c}+\overline{a}c$ \cite{A79}:
\begin{equation}
\begin{split}
 & 2 [ (\sum_{i=1}^m \re(x_i)a_i)(\sum_{i=1}^m \re(x_i)c_i) + (\sum_{i=1}^m \im(x_i)a_i)(\sum_{i=1}^m \im(x_i)a_i) ] \\
 & = 2 [ \re(\sum_{i=1}^m x_ia_i)\re(\sum_{i=1}^m x_ic_i) + \im(\sum_{i=1}^m x_ia_i)\im(\sum_{i=1}^m x_ic_i) ] \\
 & = A_1\overline{C_1} + \overline{A_1}C_1 \\
\end{split}
\end{equation}
and the theorem is proved.
\end{proof}
\begin{theorem}
\begin{equation}
\begin{split}
 & \multsum \left[ \prod_{i=1}^m \binom{a_i}{k_i}\binom{k_i}{c_i} \right] \sum_{i=1}^m x_i k_i^2 
  =  \binom{A_0-C_0}{n-C_0} \left[ \prod_{i=1}^m \binom{a_i}{c_i} \right] \\
 & \cdot \frac{ (n-C_0)[(n-C_0-1)A_{1,2}+(A_0-n)(A_1-C_1+2S_{1,1})] +(A_0-n)(A_0-n-1)C_{1,2} }{(A_0-C_0)(A_0-C_0-1)} \\
\end{split}
\end{equation}
\end{theorem}
\begin{proof}
Changing the order of summation as in (\ref{multsumorder}) and using theorem \ref{theorem2}.
\begin{equation}
\begin{split}
 & \sum_{i=1}^m x_i \{ (n-C_0)[(n-C_0-1)a_i^2 + (A_0-n)(a_i-c_i+2a_ic_i) ] + (A_0-n)(A_0-n-1)c_i^2 \} \\
 & = (n-C_0)[(n-C_0-1)A_{1,2} + (A_0-n)(A_1-C_1+2S_{1,1}) ] + (A_0-n)(A_0-n-1) C_{1,2} \\
\end{split}
\end{equation}
\end{proof}
\begin{theorem}\label{theorem3}
For integer $m\geq 1$ and $1\leq p\leq m$:
\begin{equation}
\begin{split}
 & \multsum \left[\prod_{i=1}^m \binom{a_i}{k_i}\binom{k_i}{c_i}\right] k_p^3 \\
 & = \binom{A_0-C_0}{n-C_0} \left[\prod_{i=1}^m \binom{a_i}{c_i}\right] \frac{1}{(A_0-C_0)(A_0-C_0-1)(A_0-C_0-2)} \\
 & \cdot \{ (n-C_0)(n-C_0-1) [ (n-C_0-2)a_p^3 + (A_0-n)(3a_p^2c_p+3a_p^2-3a_pc_p-a_p+c_p) ] \\
 & + (A_0-n)(A_0-n-1) [ (A_0-n-2)c_p^3 + (n-C_0)(3a_pc_p^2-3c_p^2+3a_pc_p+a_p-c_p) ] \} \\
\end{split}
\end{equation}
\end{theorem}
\begin{proof}
The following modified infinite geometric series \cite{K19} is used:
\begin{equation}
 \sum_{k=0}^{\infty} k^3 w^k = \frac{w}{(1-w)^2} + \frac{6w^2}{(1-w)^3} + \frac{6w^3}{(1-w)^4}
\end{equation}
The first two terms gives a result identical to theorem \ref{theorem1} and \ref{theorem2},
and for the third term:
\begin{equation}
\begin{split}
 &  \left[ \prod_{i=1}^m \res{w_i} \right] \frac{1}{w_m^n}
    \left[\prod_{i=1}^{m}(1+w_i)^{a_i-c_i}w_i^{c_i-1}\right] 
     \frac{6(\frac{\textstyle w_m}{\textstyle w_p})^3}{(1-\frac{\textstyle w_m}{\textstyle w_p})^4}
    \mulprod{i=1}{i\neq p}{m-1} \frac{1}{1-\frac{\textstyle w_m}{\textstyle w_i}} \\
 & = \res{w_m} \frac{(1+w_m)^{a_m-c_m}}{w_m^{n-c_m-2}} \res{w_p} \frac{6(1+w_p)^{a_p-c_p}w_p^{c_p}}{(w_p-w_m)^4}
     \mulprod{i=1}{i\neq p}{m-1} (1+w_m)^{a_i-c_i} w_m^{c_i} \\
\end{split}
\end{equation}
The residue of $w_p$ is evaluated with (\ref{resdef}) using the product
rule for the third derivative:
\begin{equation}
\begin{split}
 & \res{w_p} \frac{6(1+w_p)^{a_p-c_p}w_p^{c_p}}{(w_p-w_m)^4} \\
 & = (a_p-c_p)(a_p-c_p-1)(a_p-c_p-2)(1+w_m)^{a_p-c_p-3}w_m^{c_p} \\
 & \quad + 3(a_p-c_p)(a_p-c_p-1)c_p(1+w_m)^{a_p-c_p-2}w_m^{c_p-1} \\
 & \quad + 3(a_p-c_p)c_p(c_p-1)(1+w_m)^{a_p-c_p-1}w_m^{c_p-2} \\
 & \quad + c_p(c_p-1)(c_p-2)(1+w_m)^{a_p-c_p}w_m^{c_p-3} \\
\end{split}
\end{equation}
Using the absorption identity \cite{GKP94,K15} as in theorem \ref{theorem2},
and adding as mentioned the result of theorem \ref{theorem1} and \ref{theorem2}:
\begin{equation}
\begin{split}
 & \binom{A_0-C_0}{n-C_0} \frac{1}{(A_0-C_0)(A_0-C_0-1)(A_0-C_0-2)} \\
 & \cdot \{(n-C_0)(n-C_0-1)(a_p-c_p)(a_p-c_p-1) \\
 & \qquad \cdot [ (n-C_0-2)(a_p-c_p-2)+3(A_0-C_0-2)c_p ] \\
 & \quad + (A_0-C_0-1)(A_0-C_0-2)c_p(c_p-1) \\
 & \qquad \cdot [ (A_0-C_0)(c_p-2) + 3(n-C_0)(a_p-c_p) ] \\
 & \quad + (A_0-C_0-2) [ 3(n-C_0)(a_p-c_p)((n-C_0-1)(a_p-c_p-1)+2(A_0-C_0-1)c_p) \\
 & \qquad + (A_0-C_0-1) ( 3(A_0-C_0)c_p(c_p-1)+(n-C_0)a_p+(A_0-n)c_p) ] \} \\
\end{split}
\end{equation}
This result simplifies to the theorem.
\end{proof}
\begin{theorem}\label{theorem12}
For integer $m\geq 2$, $1\leq p\leq m$, $1\leq q\leq m$ and $p\neq q$:
\begin{equation}
\begin{split}
 & \multsum \left[\prod_{i=1}^m \binom{a_i}{k_i}\binom{k_i}{c_i}\right] k_p k_q^2 \\ 
 &  = \binom{A_0-C_0}{n-C_0} \left[\prod_{i=1}^m \binom{a_i}{c_i}\right] \frac{1}{(A_0-C_0)(A_0-C_0-1)(A_0-C_0-2)} \\
 & \cdot \{ (n-C_0)(n-C_0-1) [ (n-C_0-2)a_pa_q^2 + (A_0-n)(c_pa_q^2+a_pa_q-a_pc_q+2a_pa_qc_q) ] \\
 & + (A_0-n)(A_0-n-1) [ (A_0-n-2)c_pc_q^2 + (n-C_0)(a_pc_q^2-c_pc_q+c_pa_q+2c_pa_qc_q) ] \} \\
\end{split}
\end{equation}
\end{theorem}
\begin{proof}
The proof is similar to theorem \ref{theorem11} but with two quadratic and one cubic residues
instead of two quadratic residues.
\begin{equation}
\begin{split}
 & \res{w_m} \frac{(1+w_m)^{a_m-c_m}}{w_m^{n-c_m-2}} \res{w_p} \frac{(1+w_p)^{a_p-c_p}w_p^{c_p}}{(w_p-w_m)^2} \\
 & \cdot  [ \res{w_q} \frac{(1+w_q)^{a_p-c_p}w_q^{c_p}}{(w_q-w_m)^2} + \res{w_q} \frac{2(1+w_q)^{a_p-c_p}w_q^{c_p}}{(w_q-w_m)^3} ]
  \mulprod{i=1}{i\neq p, i\neq q}{m-1} (1+w_m)^{a_i-c_i} w_m^{c_i} \\
\end{split}
\end{equation}
The first product gives a result identical to theorem \ref{theorem11},
and for the second product the residues are given by (\ref{resquad}), and multiplying these out gives:
\begin{equation}
\begin{split}
 & [(a_p-c_p)(1+w_m)^{a_p-c_p-1}w_m^{c_p} + c_p(1+w_m)^{a_p-c_p}w_m^{c_p-1} ] \\
 & \cdot [(a_q-c_q)(a_q-c_q-1)(1+w_m)^{a_q-c_q-2}w_m^{c_q} + 2(a_q-c_q)c_q(1+w_m)^{a_q-c_q-1}w_m^{c_q-1} \\
 & \qquad + c_q(c_q-1)(1+w_m)^{a_q-c_q}w_m^{c_q-2} ] \\
 & = (a_p-c_p)(a_q-c_q)(a_q-c_q-1)(1+w_m)^{a_p-c_p+a_q-c_q-3}w_m^{c_p+c_q} \\
 & + [ c_p(a_q-c_q)(a_q-c_q-1) + 2(a_p-c_p)(a_q-c_q)c_q ] (1+w_m)^{a_p-c_p+a_q-c_q-2} w_m^{c_p+c_q-1} \\
 & + [ (a_p-c_p)c_q(c_q-1) + 2c_p(a_q-c_q)c_q ] (1+w_m)^{a_p-c_p+a_q-c_q-1} w_m^{c_p+c_q-2} \\
 & + c_pc_q(c_q-1)(1+w_m)^{a_p-c_p+a_q-c_q}w_m^{c_p+c_q-3} \\
\end{split}
\end{equation}
Using the absorption identity \cite{GKP94,K15} as in theorem \ref{theorem3},
and adding as mentioned the result of theorem \ref{theorem11}:
\begin{equation}
\begin{split}
 & \binom{A_0-C_0}{n-C_0} \frac{1}{(A_0-C_0)(A_0-C_0-1)(A_0-C_0-2)} \\
 & \cdot \{(n-C_0)(n-C_0-1)(a_q-c_q) [ (n-C_0-2)(a_p-c_p)(a_q-c_q-1) \\
 & \qquad\qquad + (A_0-C_0-2)(c_p(a_q-c_q-1)+2(a_p-c_p)c_q) ] \\
 & \quad + (A_0-C_0-1)(A_0-C_0-2)c_q [ (A_0-C_0)c_p(c_q-1) \\
 & \qquad\qquad + (n-C_0)((a_p-c_p)(c_q-1) + 2c_p(a_q-c_q)) ] \\
 & \quad + (A_0-C_0-2) [ (A_0-n)(A_0-n-1)c_pc_q  \\
 & \qquad\quad + (n-C_0)((n-C_0-1)a_pa_q + (A_0-n)(a_pc_q+c_pa_q)) ] \} \\
\end{split}
\end{equation}
This result simplifies to the theorem.
\end{proof}
\begin{theorem}\label{theorem111}
For integer $m\geq 3$, $1\leq p\leq m$, $1\leq q\leq m$, $1\leq r\leq m$, $p\neq q\neq r$ and $p\neq r$:
\begin{equation}
\begin{split}
 & \multsum \left[\prod_{i=1}^m \binom{a_i}{k_i}\binom{k_i}{c_i}\right] k_p k_q k_r \\ 
 &  = \binom{A_0-C_0}{n-C_0} \left[\prod_{i=1}^m \binom{a_i}{c_i}\right] \frac{1}{(A_0-C_0)(A_0-C_0-1)(A_0-C_0-2)} \\
 & \cdot \{ (n-C_0)(n-C_0-1) [ (n-C_0-2)a_pa_qa_r + (A_0-n)(a_pa_qc_r+a_pc_qa_r+c_pa_qa_r) ] \\
 & + (A_0-n)(A_0-n-1) [ (A_0-n-2)c_pc_qc_r + (n-C_0)(c_pc_qa_r+c_pa_qc_r+a_pc_qc_r) ] \} \\
\end{split}
\end{equation}
\end{theorem}
\begin{proof}
The proof is similar to theorem \ref{theorem11} but with three quadratic residues
instead of two.
\begin{equation}
\begin{split}
 & \res{w_m} \frac{(1+w_m)^{a_m-c_m}}{w_m^{n-c_m-2}} \res{w_p} \frac{(1+w_p)^{a_p-c_p}w_p^{c_p}}{(w_p-w_m)^2}
    \res{w_q} \frac{(1+w_q)^{a_p-c_p}w_q^{c_p}}{(w_q-w_m)^2} \\
 & \qquad \cdot \res{w_r} \frac{(1+w_r)^{a_p-c_p}w_r^{c_p}}{(w_r-w_m)^2} 
   \mulprod{i=1}{i\neq p, i\neq q, i\neq r}{m-1} (1+w_m)^{a_i-c_i} w_m^{c_i} \\
\end{split}
\end{equation}
The three last residues are given by (\ref{resquad}), and multiplying these out gives:
\begin{equation}
\begin{split}
 & [(a_p-c_p)(1+w_m)^{a_p-c_p-1}w_m^{c_p} + c_p(1+w_m)^{a_p-c_p}w_m^{c_p-1} ] \\
 & \cdot [(a_q-c_q)(1+w_m)^{a_q-c_q-1}w_m^{c_q} + c_q(1+w_m)^{a_q-c_q}w_m^{c_q-1} ] \\
 & \cdot [(a_r-c_r)(1+w_m)^{a_r-c_r-1}w_m^{c_r} + c_r(1+w_m)^{a_r-c_r}w_m^{c_r-1} ] \\
 & = (a_p-c_p)(a_q-c_q)(a_r-c_r)(1+w_m)^{a_p-c_p+a_q-c_q+a_r-c_r-3}w_m^{c_p+c_q+c_r} \\
 & \quad + [ (a_p-c_p)(a_q-c_q)c_r + (a_p-c_p)c_q(a_r-c_r) + c_p(a_q-c_q)(a_r-c_r) ] \\
 & \qquad \cdot (1+w_m)^{a_p-c_p+a_q-c_q+a_r-c_r-2}w_m^{c_p+c_q+c_r-1} \\
 & \quad + [ (a_p-c_p)c_qc_r + c_p(a_q-c_q)c_r + c_pc_q(a_r-c_r) ] \\
 & \qquad \cdot (1+w_m)^{a_p-c_p+a_q-c_q+a_r-c_r-1}w_m^{c_p+c_q+c_r-2} \\
 & \quad + c_pc_qc_r(1+w_m)^{a_p-c_p+a_q-c_q+a_r-c_r}w_m^{c_p+c_q+c_r-3} \\
\end{split}
\end{equation}
Using the absorption identity \cite{GKP94,K15} as in theorem \ref{theorem12}:
\begin{equation}
\begin{split}
 & \binom{A_0-C_0}{n-C_0} \frac{1}{(A_0-C_0)(A_0-C_0-1)(A_0-C_0-2)} \\
 & \cdot \{ (n-C_0)(n-C_0-1) [(n-C_0-2)(a_p-c_p)(a_q-c_q)(a_r-c_r) + (A_0-C_0-2) \\
 & \qquad\qquad \cdot ((a_p-c_p)(a_q-c_q)c_r+(a_p-c_p)c_q(a_r-c_r)+c_p(a_q-c_q)(a_r-c_r)) ] \\
 & \qquad + (A_0-C_0-1)(A_0-C_0-2) [ (A_0-C_0) c_pc_qc_r \\
 & \qquad\qquad + (n-C_0)((a_p-c_p)c_qc_r+c_p(a_q-c_q)c_r+c_pc_q(a_r-c_r)) ] \} \\
\end{split}
\end{equation}
This result simplifies to the theorem.
\end{proof}
\begin{theorem}
For integer $m\geq 1$:
\begin{equation}
\begin{split}
 & \multsum \left[ \prod_{i=1}^m \binom{a_i}{k_i}\binom{k_i}{c_i} \right] (\sum_{i=1}^m x_i k_i)^3 \\
 & =  \binom{A_0-C_0}{n-C_0} \left[ \prod_{i=1}^m \binom{a_i}{c_i} \right] \frac{1}{(A_0-C_0)(A_0-C_0-1)(A_0-C_0-2)} \\
 & \qquad\cdot \{ (n-C_0)(n-C_0-1) [ (n-C_0-2)A_1^3 \\ 
 & \qquad\qquad\qquad +(A_0-n)(C_3-A_3+3A_1(A_2-C_2+A_1C_1))] \\
 & \qquad\quad +(A_0-n)(A_0-n-1) [ (A_0-n-2)C_1^3 \\
 & \qquad\qquad\qquad + (n-C_0)(A_3-C_3+3C_1(A_2-C_2+A_1C_1))] \} \\
\end{split}
\end{equation}
\end{theorem}
\begin{proof}
The product of the sum in the summand is expressed in sums with unequal indices:
\begin{equation}
 (\sum_{i=1}^m x_ik_i)^3 = \tripsum{p}{q}{r} x_px_qx_r k_pk_qk_r
 + 3\dblsum{p}{q} x_px_q^2 k_pk_q^2 + \sum_{p=1}^m x_p^3 k_p^3
\end{equation}
Changing the order of summation, the multiple sum over the products of $k_p$, $k_q$ and $k_r$ is
substituted from theorems \ref{theorem111}, \ref{theorem12} and \ref{theorem3}:
\begin{equation}
\begin{split}
 & \tripsum{i}{j}{k} x_ix_jx_k \{ (n-C_0)(n-C_0-1) [ (n-C_0-2)a_ia_ja_k + 3(A_0-n)a_ia_jc_k ] \\
 & \qquad\qquad + (A_0-n)(A_0-n-1) [ (A_0-n-2)c_ic_jc_k + 3(n-C_0)c_ic_ja_k ] \} \\
 & + 3 \dblsum{i}{j} x_ix_j^2 \{ (n-C_0)(n-C_0-1) [ (n-C_0-2)a_ia_j^2 \\ 
 & \qquad\qquad\qquad\qquad + (A_0-n)(c_ia_j^2+a_ia_j-a_ic_j+2a_ia_jc_j) ] \\
 & \qquad\qquad + (A_0-n)(A_0-n-1) [ (A_0-n-2)c_ic_j^2 \\
 & \qquad\qquad\qquad\qquad + (n-C_0)(a_ic_j^2-c_ic_j+c_ia_j+2c_ia_jc_j) ] \} \\
 & + \sum_{i=1}^m x_i^3 \{ (n-C_0)(n-C_0-1) [ (n-C_0-2)a_i^3 + (A_0-n)(3a_i^2c_i+3a_i^2-3a_ic_i-a_i+c_i) ] \\
 & \qquad + (A_0-n)(A_0-n-1) [ (A_0-n-2)c_i^3 + (n-C_0)(3a_ic_i^2-3c_i^2+3a_ic_i+a_i-c_i) ] \} \\
\end{split}
\end{equation}
The expresions of $A_1^3$, $C_1^3$, $A_1^2C_1$ and $A_1C_1^2$ as sums over
unequal indices are given by equation (\ref{tripexpand}):
\begin{equation}
\begin{split}
 & (n-C_0)(n-C_0-1) [ (n-C_0-2)A_1^3 +3(A_0-n)A_1^2C_1 ] \\
 & + (A_0-n)(A_0-n-1) [ (A_0-n-2)C_1^3 + 3(n-C_0)A_1C_1^2 ] \\
 & + 3 \dblsum{i}{j} x_ix_j^2 [ (n-C_0)(n-C_0-1)(A_0-n)(a_ia_j-a_ic_j) \\
 & \qquad \qquad + (A_0-n)(A_0-n-1)(n-C_0)(c_ia_j-c_ic_j) ] \\
 & + \sum_{i=1}^m x_i^3 [ (n-C_0)(n-C_0-1)(A_0-n)(3a_i^2-3a_ic_i-a_i+c_i) \\
 & \qquad\qquad + (A_0-n)(A_0-n-1)(n-C_0)(-3c_i^2+3a_ic_i+a_i-c_i) ] \\
\end{split}
\end{equation}
The remaining sums can be expressed in $A_1$, $C_1$, $A_2$, $C_2$, $A_3$ and $C_3$:
\begin{equation}
\begin{split}
 & (n-C_0)(n-C_0-1) [ (n-C_0-2)A_1^3 +3(A_0-n)A_1^2C_1 ] \\
 & + (A_0-n)(A_0-n-1) [ (A_0-n-2)C_1^3 + 3(n-C_0)A_1C_1^2 ] \\
 & + 3(n-C_0)(A_0-n) [ (n-C_0-1)(A_1A_2-A_1C_2) \\
 & \qquad \qquad + (A_0-n-1)(C_1A_2-C_1C_2)] \\
 & + (n-C_0)(A_0-n) [ (n-C_0-1) - (A_0-n-1) ] ( C_3-A_3 ) \\
\end{split}
\end{equation}
This result simplifies to the theorem.
\end{proof}
\begin{theorem}
For integer $m\geq 1$:
\begin{equation}
\begin{split}
 & \multsum \left[ \prod_{i=1}^m \binom{a_i}{k_i}\binom{k_i}{c_i} \right] \sum_{i=1}^m x_i k_i^3 \\
 & =  \binom{A_0-C_0}{n-C_0} \left[ \prod_{i=1}^m \binom{a_i}{c_i} \right] \frac{1}{(A_0-C_0)(A_0-C_0-1)(A_0-C_0-2)} \\
 & \qquad\cdot \{ (n-C_0)(n-C_0-1) [ (n-C_0-2)A_{1,3} \\ 
 & \qquad\qquad\qquad +(A_0-n)(C_1-A_1+3(S_{2,1}+A_{1,2}-S_{1,1}))] \\
 & \qquad\quad +(A_0-n)(A_0-n-1) [ (A_0-n-2)C_{1,3} \\
 & \qquad\qquad\qquad + (n-C_0)(A_1-C_1+3(S_{1,2}-C_{1,2}+S_{1,1}))] \} \\
\end{split}
\end{equation}
\end{theorem}
\begin{proof}
Changing the order of summation as in (\ref{multsumorder}) and using theorem \ref{theorem3}.
\begin{equation}
\begin{split}
 & \sum_{i=1}^m x_i \{ (n-C_0)(n-C_0-1) [ (n-C_0-2)a_i^3 + (A_0-n)(3a_i^2c_i+3a_i^2-3a_ic_i-a_i+c_i) ] \\
 & \qquad + (A_0-n)(A_0-n-1) [ (A_0-n-2)c_i^3 + (n-C_0)(3a_ic_i^2-3c_i^2+3a_ic_i+a_i-c_i) ] \} \\
 & = (n-C_0)(n-C_0-1) [ (n-C_0-2)A_{1,3} +( A_0-n)(3S_{2,1}+3A_{1,2}-3S_{1,1}-A_1+C_1)] \\ 
 & +(A_0-n)(A_0-n-1) [ (A_0-n-2)C_{1,3}+ (n-C_0)(3S_{1,2}-3C_{1,2}+3S_{1,1}+A_1-C_1)]  \\
\end{split}
\end{equation}
which results in the theorem.
\end{proof}

\section{Proof of the Unrestricted Multi-sum Identities}

\begin{theorem}\label{simple0}
For integer $m\geq 1$:
\begin{equation}
 \multsumfree \prod_{i=1}^m \binom{a_i}{k_i}\binom{k_i}{c_i} = 2^{A_0-C_0} \prod_{i=1}^m \binom{a_i}{c_i}
\end{equation}
\end{theorem}
\begin{proof}
The multiple sum of the product on the left side is a product of sums:
\begin{equation}
 \multsumfree \prod_{i=1}^m \binom{a_i}{k_i}\binom{k_i}{c_i} = \prod_{i=1}^m \sum_{k_i=0}^{a_i} \binom{a_i}{k_i}\binom{k_i}{c_i}
\end{equation}
The trinomial revision identity \cite{GKP94,K15} is applied as in theorem \ref{theorem0}:
\begin{equation}
 \binom{a_i}{k_i}\binom{k_i}{c_i} = \binom{a_i-c_i}{k_i-c_i}\binom{a_i}{c_i}
\end{equation}
Because this expression is zero when $c_i>a_i$,
in the individual sums it can be assumed that $0\leq c_i\leq a_i$.
Then each sum can be evaluated as follows, where the upper bound $a$ can be
replaced by $\infty$ because the summand is zero when $k>a$:
\begin{equation}
\begin{split}
 & \sum_{k=0}^a \binom{a-c}{k-c} = \sum_{k=0}^{\infty} \res{w} \frac{(1+w)^{a-c}}{w^{k-c+1}}
 = \res{w} (1+w)^{a-c}w^{c-1} \sum_{k=0}^{\infty} (\frac{1}{w})^k \\
 & = \res{w} \frac{(1+w)^{a-c}w^{c-1}}{1-1/w} = \res{w} \frac{(1+w)^{a-c}w^c}{w-1} = 2^{a-c} \\
\end{split}
\end{equation}
The product of these individual sums gives the theorem.
\end{proof}
\begin{theorem}\label{simple1}
For integer $m\geq 1$ and $1\leq p\leq m$:
\begin{equation}
 \multsumfree \left[ \prod_{i=1}^m \binom{a_i}{k_i}\binom{k_i}{c_i} \right] k_p 
 = 2^{A_0-C_0-1} \left[ \prod_{i=1}^m \binom{a_i}{c_i} \right] (a_p+c_p)
\end{equation}
\end{theorem}
\begin{proof}
The following modified infinite geometric series \cite{K19} is used:
\begin{equation}
 \sum_{k=0}^{\infty} k w^k = \frac{w}{(1-w)^2}
\end{equation}
One of the individual sums of the product becomes:
\begin{equation}
\begin{split}
 & \sum_{k=0}^a \binom{a-c}{k-c} k = \sum_{k=0}^a k \res{w} \frac{(1+w)^{a-c}}{w^{k-c+1}}
 = \res{w} (1+w)^{a-c}w^{c-1} \sum_{k=0}^{\infty} k (\frac{1}{w})^k \\
 & = \res{w} \frac{(1+w)^{a-c}w^{c-1}}{w(1-1/w)^2} = \res{w} \frac{(1+w)^{a-c}w^c}{(w-1)^2} = (a-c)2^{a-c-1}+c 2^{a-c} \\
 & = 2^{a-c-1}(a-c+2c) = 2^{a-c-1}(a+c) \\
\end{split}
\end{equation}
The other individual sums of the product remain identical,
so the product becomes the theorem.
\end{proof}
\begin{theorem}
For integer $m\geq 1$:
\begin{equation}
 \multsumfree \left[ \prod_{i=1}^m \binom{a_i}{k_i}\binom{k_i}{c_i} \right] \sum_{i=1}^m x_i k_i
  = 2^{A_0-C_0-1} \left[ \prod_{i=1}^m \binom{a_i}{c_i} \right] (A_1+C_1)
\end{equation}
\end{theorem}
\begin{proof}
Changing the order of summation and applying theorem \ref{simple1}: 
\begin{equation}
 \sum_{i=1}^m x_i (a_i+c_i) = A_1 + C_1
\end{equation}
\end{proof}
\begin{theorem}\label{simple2}
For integer $m\geq 1$ and $1\leq p\leq m$:
\begin{equation}
 \multsumfree \left[ \prod_{i=1}^m \binom{a_i}{k_i}\binom{k_i}{c_i} \right] k_p^2
 = 2^{A_0-C_0-2} \left[ \prod_{i=1}^m \binom{a_i}{c_i} \right] [(a_p+c_p)^2+a_p-c_p]
\end{equation}
\end{theorem}
\begin{proof}
The following modified infinite geometric series \cite{K19} is used:
\begin{equation}
 \sum_{k=0}^{\infty} k^2 w^k = \frac{w}{(1-w)^2} + \frac{2w^2}{(1-w)^3}
\end{equation}
One of the individual sums of the product becomes:
\begin{equation}
\begin{split}
 & \sum_{k=0}^a \binom{a-c}{k-c} k^2 = \res{w} (1+w)^{a-c}w^{c-1} \sum_{k=0}^{\infty} k^2 (\frac{1}{w})^k \\
 & = \res{w} (1+w)^{a-c}w^{c-1} [\frac{1}{w(1-1/w)^2} + \frac{2}{w^2(1-1/w)^3} ] \\
 & = \res{w} (1+w)^{a-c}w^c [ \frac{1}{(w-1)^2} + \frac{2}{(w-1)^3} ] \\
 & = 2^{a-c-1}(a+c) + (a-c)(a-c-1)2^{a-c-2} + 2(a-c)c 2^{a-c-1} + c(c-1)2^{a-c} \\ 
 & = 2^{a-c-2} [ 2(a+c)+(a-c)(a-c-1)+4(a-c)c+4c(c-1) ] \\
 & = 2^{a-c-2} [ (a+c)^2 + a-c ] \\ 
\end{split}
\end{equation}
The other individual sums of the product remain identical,
so the product becomes the theorem.
\end{proof}
\begin{theorem}\label{simple11}
For integer $m\geq 1$, $1\leq p\leq m$, $1\leq q\leq m$ and $p\neq q$:
\begin{equation}
 \multsumfree \left[ \prod_{i=1}^m \binom{a_i}{k_i}\binom{k_i}{c_i} \right] k_p k_q
 = 2^{A_0-C_0-2} \left[ \prod_{i=1}^m \binom{a_i}{c_i} \right] (a_p+c_p)(a_q+c_q)
\end{equation}
\end{theorem}
\begin{proof}
This theorem follows from applying theorem \ref{simple1} twice.
\end{proof}
\begin{theorem}
For integer $m\geq 1$:
\begin{equation}
\begin{split}
 & \multsumfree \left[ \prod_{i=1}^m \binom{a_i}{k_i}\binom{k_i}{c_i} \right] (\sum_{i=1}^m x_i k_i)(\sum_{i=1}^m y_i k_i) \\
 & = 2^{A_0-C_0-2} \left[ \prod_{i=1}^m \binom{a_i}{c_i} \right] [ A_{1,1}^*-C_{1,1}^* + (A_1+C_1)(A_1^*+C_1^*) ] \\
\end{split}
\end{equation}
\end{theorem}
\begin{proof}
The product in the summand becomes:
\begin{equation}
 (\sum_{i=1}^m x_i k_i)(\sum_{i=1}^m y_i k_i) 
  = \dblsum{i}{j} x_i y_j k_i k_j 
   + \sum_{i=1}^m x_iy_ik_i^2 
\end{equation}
Changing the order of summation and using theorem \ref{simple2} and \ref{simple11}.
\begin{equation}
\begin{split}
 & \dblsum{i}{j} x_iy_j(a_i+c_i)(a_j+c_j) + \sum_{i=1}^m x_iy_i [(a_i+c_i)^2+a_i-c_i] \\
 & = \sum_{i=1}^m \sum_{j=1}^m x_iy_j(a_i+c_i)(a_j+c_j) + \sum_{i=1}^m x_iy_i(a_i-c_i) \\
 & = A_{1,1}^*-C_{1,1}^* + (A_1+C_1)(A_1^*+C_1^*) \\
\end{split}
\end{equation}
\end{proof}
\begin{theorem}
For integer $m\geq 1$:
\begin{equation}
\begin{split}
 & \multsumfree \left[ \prod_{i=1}^m \binom{a_i}{k_i}\binom{k_i}{c_i} \right] (\sum_{i=1}^m x_i k_i)^2 \\
 & = 2^{A_0-C_0-2} \left[ \prod_{i=1}^m \binom{a_i}{c_i} \right] [ A_2-C_2 + (A_1+C_1)^2 ] \\
\end{split}
\end{equation}
\end{theorem}
\begin{proof}
In the previous theorem taking $y_i=x_i$, then $A_1^*=A_1$, $C_1^*=C_1$, $A_{1,1}^*=A_2$ and $C_{1,1}^*=C_2$
gives this theorem.
\end{proof}
\begin{theorem}
For integer $m\geq 1$:
\begin{equation}
\begin{split}
 & \multsumfree \left[ \prod_{i=1}^m \binom{a_i}{k_i}\binom{k_i}{c_i} \right] |\sum_{i=1}^m x_i k_i|^2 \\
 & = 2^{A_0-C_0-2} \left[ \prod_{i=1}^m \binom{a_i}{c_i} \right] ( \Aabs-\Cabs + |A_1+C_1|^2 ) \\
\end{split}
\end{equation}
\end{theorem}
\begin{proof}
Using the previous theorem and using the same method as in theorem \ref{theoremabs}
and using:
\begin{equation}
 [\re(A_1)+\re(C_1)]^2 + [\im(A_1)+\im(C_1)]^2 = \re^2(A_1+C_1) + \im^2(A_1+C_1) = |A_1+C_1|^2
\end{equation}
gives the theorem.
\end{proof}
\begin{theorem}
For integer $m\geq 1$:
\begin{equation}
\begin{split}
 & \multsumfree \left[ \prod_{i=1}^m \binom{a_i}{k_i}\binom{k_i}{c_i} \right] \sum_{i=1}^m x_i k_i^2 \\
 & = 2^{A_0-C_0-2} \left[ \prod_{i=1}^m \binom{a_i}{c_i} \right] ( A_1-C_1+A_{1,2}+C_{1,2}+2S_{1,1} ) \\
\end{split}
\end{equation}
\end{theorem}
\begin{proof}
Changing the order of summation and using theorem \ref{simple2}.
\begin{equation}
\begin{split}
 & \sum_{i=1}^m x_i [ (a_i+c_i)^2 + a_i - c_i ] = \sum_{i=1}^m x_i ( a_i^2+c_i^2+2a_ic_i+a_i-c_i ) \\
 & = A_{1,2}+C_{1,2}+2S_{1,1}+A_1-C_1 \\
\end{split}
\end{equation}
\end{proof}
\begin{theorem}\label{simple3}
For integer $m\geq 1$ and $1\leq p\leq m$:
\begin{equation}
  \multsumfree \left[ \prod_{i=1}^m \binom{a_i}{k_i}\binom{k_i}{c_i} \right] k_p^3 
  = 2^{A_0-C_0-3} \left[ \prod_{i=1}^m \binom{a_i}{c_i} \right] (a_p+c_p)[(a_p+c_p)^2 + 3(a_p-c_p) ] 
\end{equation}
\end{theorem}
\begin{proof}
The following modified infinite geometric series \cite{K19} is used:
\begin{equation}
 \sum_{k=0}^{\infty} k^3 w^k = \frac{w}{(1-w)^2} + \frac{6w^2}{(1-w)^3} + \frac{6w^3}{(1-w)^4}
\end{equation}
One of the individual sums of the product becomes:
\begin{equation}
\begin{split}
 & \sum_{k=0}^a \binom{a-c}{k-c} k^3 = \res{w} (1+w)^{a-c}w^{c-1} \sum_{k=0}^{\infty} k^3 (\frac{1}{w})^k \\
 & = \res{w} (1+w)^{a-c}w^{c-1} [\frac{1}{w(1-1/w)^2} + \frac{6}{w^2(1-1/w)^3} + \frac{6}{w^3(1-1/w)^4} ] \\
 & = \res{w} (1+w)^{a-c}w^c [ \frac{1}{(w-1)^2} + \frac{6}{(w-1)^3} + \frac{6}{(w-1)^4} ] \\
 & = 2^{a-c-1}(a+c) + 3 [ (a-c)(a-c-1)2^{a-c-2} + 2(a-c)c 2^{a-c-1} + c(c-1)2^{a-c} ] \\
 & \quad + (a-c)(a-c-1)(a-c-2)2^{a-c-3} + 3(a-c)(a-c-1)c2^{a-c-2} \\
 & \quad + 3(a-c)c(c-1)2^{a-c-1} + c(c-1)(c-2)2^{a-c} \\ 
 & = 2^{a-c-3} (a+c)[(a+c)^2+3(a-c)] \\
\end{split}
\end{equation}
The other individual sums of the product remain identical,
so the product becomes the theorem.
\end{proof}
\begin{theorem}\label{simple12}
For integer $m\geq 1$, $1\leq p\leq m$, $1\leq q\leq m$ and $p\neq q$:
\begin{equation}
  \multsumfree \left[ \prod_{i=1}^m \binom{a_i}{k_i}\binom{k_i}{c_i} \right] k_p k_q^2 
  = 2^{A_0-C_0-3} \left[ \prod_{i=1}^m \binom{a_i}{c_i} \right] (a_p+c_p)[(a_q+c_q)^2 + a_q-c_q ] 
\end{equation}
\end{theorem}
\begin{proof}
This theorem follows from applying theorem \ref{simple1} and \ref{simple2}.
\end{proof}
\begin{theorem}\label{simple111}
For integer $m\geq 3$, $1\leq p\leq m$, $1\leq q\leq m$, $1\leq r\leq m$, $p\neq q\neq r$ and $p\neq r$:
\begin{equation}
 \multsumfree \left[ \prod_{i=1}^m \binom{a_i}{k_i}\binom{k_i}{c_i} \right] k_p k_q k_r 
  = 2^{A_0-C_0-3} \left[ \prod_{i=1}^m \binom{a_i}{c_i} \right] (a_p+c_p)(a_q+c_q)(a_r+c_r)
\end{equation}
\end{theorem}
\begin{proof}
This theorem follows from applying theorem \ref{simple1} three times.
\end{proof}
\begin{theorem}
For integer $m\geq 1$:
\begin{equation}
\begin{split}
 & \multsumfree \left[ \prod_{i=1}^m \binom{a_i}{k_i}\binom{k_i}{c_i} \right] (\sum_{i=1}^m x_i k_i)^3 \\
 & = 2^{A_0-C_0-3} \left[ \prod_{i=1}^m \binom{a_i}{c_i} \right] (A_1+C_1)[(A_1+C_1)^2 + 3(A_2-C_2)] \\
\end{split}
\end{equation}
\end{theorem}
\begin{proof}
Changing the order of summation and using theorem \ref{simple3}, \ref{simple12} and \ref{simple111}.
\begin{equation}
\begin{split}
 & \tripsum{i}{j}{k} x_ix_jx_k (a_i+c_i)(a_j+c_j)(a_k+c_k) \\
 & \quad + 3 \dblsum{i}{j} x_ix_j^2 (a_i+c_i)((a_j+c_j)^2+a_j-c_j) \\
 & \quad + \sum_{i=1}^m x_i^3 (a_i+c_i)[(a_i+c_i)^2+3(a_i-c_i)] \\
 & = (A_1+C_1)^3 + 3 \dblsum{i}{j} x_ix_j^2 (a_i+c_i)(a_j-c_j ) 
  + 3 \sum_{i=1}^m x_i^3 (a_i+c_i)(a_i-c_i) \\
 & = (A_1+C_1)^3 + 3(A_1+C_1)(A_2-C_2) = (A_1+C_1)[(A_1+C_1)^2+3(A_2-C_2)] \\
\end{split}
\end{equation}
\end{proof}
\begin{theorem}
For integer $m\geq 1$:
\begin{equation}
\begin{split}
 & \multsumfree \left[ \prod_{i=1}^m \binom{a_i}{k_i}\binom{k_i}{c_i} \right] \sum_{i=1}^m x_i k_i^3 \\
 & = 2^{A_0-C_0-3} \left[ \prod_{i=1}^m \binom{a_i}{c_i} \right] [ A_{1,3}+C_{1,3}+3(A_{1,2}-C_{1,2}+S_{1,2}+S_{2,1}) ] \\
\end{split}
\end{equation}
\end{theorem}
\begin{proof}
Changing the order of summation and using theorem \ref{simple3}.
\begin{equation}
\begin{split}
 & \sum_{i=1}^m x_i (a_i+c_i) [ (a_i+c_i)^2 + 3(a_i-c_i) ] \\
 & = \sum_{i=1}^m x_i ( a_i^3+c_i^3+3a_i^2-3c_i^2+3a_ic_i^2+3a_i^2c_i) \\
 & = A_{1,3}+C_{1,3}+3(A_{1,2}-C_{1,2}+S_{1,2}+S_{2,1}) \\
\end{split}
\end{equation}
\end{proof}

\section{A Proof of Residue Identity \ref{resdef}}\label{ressect}

\begin{theorem}
When a complex function $f(z)$ has a pole of order $m$ at $z=z_p$,
then the residue of this function at $z=z_p$ is \cite{K08}: 
\begin{equation}
 \res{z=z_p} f(z) = \frac{1}{(m-1)!} D_z^{m-1} [(z-z_p)^m f(z)] |_{z=z_p} 
\end{equation}
where $D_z^n f(z)|_{z=z_p}$ is the $n$-th derivative of $f(z)$ at $z=z_p$.
\end{theorem}
\begin{proof}
When a complex function $f(z)$ has a pole of order $m$ at $z=z_p$,
then the residue of $f(z)$ at $z=z_p$ can be defined as the
coefficient $a_{-1}$ in the Laurent series expansion of $f(z)$ at $z=z_p$,
where $a_{-m}\neq 0$ \cite{K08,RW00}:
\begin{equation}
 f(z) = \sum_{k=-m}^{\infty} a_k (z-z_p)^k
\end{equation}
From this follows:
\begin{equation}
 (z-z_p)^m f(z) = \sum_{k=-m}^{\infty} a_k (z-z_p)^{m+k}
 = \sum_{k=0}^{\infty} a_{k-m} (z-z_p)^k
\end{equation}
For integer $k\geq 0$:
\begin{equation}
 D_z^n (z-z_p)^k = 
\begin{cases}
 (z-z_p)^{k-n} \prod_{j=0}^{n-1} (k-j) & \text{if $k\geq n$} \\
 0 & \text{if $k<n$} \\
\end{cases}
\end{equation}
The $n$-th derivative of $(z-z_p)^m f(z)$ at $z=z_p$ becomes,
by differentiating the infinite power series term by term \cite{Knopp}, 
and using $0^0=1$ \cite{GKP94}, where $\delta_{k,n}$ is the Kronecker delta:
\begin{equation}
\begin{split}
 & \frac{1}{n!} D_z^n \sum_{k=0}^{\infty} a_{k-m} (z-z_p)^k |_{z=z_p} 
  = \frac{1}{n!} \sum_{k=0}^{\infty} a_{k-m} D_z^n (z-z_p)^k |_{z=z_p} \\
 & = \frac{1}{n!} \sum_{k=n}^{\infty} a_{k-m} (z-z_p)^{k-n} \prod_{j=0}^{n-1} (k-j) |_{z=z_p} 
  = \frac{1}{n!} \sum_{k=n}^{\infty} a_{k-m} 0^{k-n} \prod_{j=0}^{n-1} (k-j) \\
 & = \frac{1}{n!} \sum_{k=n}^{\infty} a_{k-m} \delta_{k,n} \prod_{j=0}^{n-1} (k-j) 
  = \frac{1}{n!} a_{n-m} \prod_{j=0}^{n-1} (n-j) = \frac{1}{n!} a_{n-m} n! = a_{n-m} \\
\end{split}
\end{equation}
Taking $n=m-1$ the theorem is proved.
\end{proof}
\begin{theorem}
When a complex function $f(z)$ does not have a pole at $z=z_p$, then:
\begin{equation}
 \res{z=z_p} \frac{f(z)}{(z-z_p)^m} = \frac{1}{(m-1)!} D_z^{m-1} f(z)|_{z=z_p}
\end{equation}
\end{theorem}
\begin{proof}
Because $f(z)$ does not have a pole at $z=z_p$,
the power series expansion of $f(z)$ at $z=z_p$ is the Taylor series expansion:
\begin{equation}
 f(z) = \sum_{k=0}^{\infty} a_k (z-z_p)^k
\end{equation}
and therefore:
\begin{equation}
 \frac{f(z)}{(z-z_p)^m} = \sum_{k=0}^{\infty} a_k (z-z_p)^{k-m}
 = \sum_{k=-m}^{\infty} a_{m+k} (z-z_p)^k
 = \sum_{k=-m}^{\infty} b_k (z-z_p)^k
\end{equation}
where $b_k=a_{m+k}$. The residue of this function at $z=z_p$ is $b_{-1}=a_{m-1}$,
and using the same method as in the previous theorem:
\begin{equation}
 \frac{1}{n!} D_z^n f(z)|_{z=z_p} = \frac{1}{n!} D_z^n \sum_{k=0}^{\infty} a_k (z-z_p)^k |_{z=z_p} = a_n
\end{equation}
Taking $n=m-1$ the theorem is proved.
\end{proof}
When in the last theorem $a_0\neq 0$, the last two theorems are equivalent.

\pdfbookmark[0]{References}{}


\begin{thebibliography}{99}
\bibitem{A79}
  L.V. Ahlfors,
  \textit{Complex Analysis},
  McGraw-Hill, 1979.
\bibitem{E84}
  G.P. Egorychev,
  \textit{Integral Representation and the Computation of Combinatorial Sums},
  Translations of Mathematical Monographs, 59, Amer. Math. Soc., 1984.
\bibitem{G72}
  H.W. Gould,
  \textit{Combinatorial Identities}, rev. ed.,
  Morgantown, 1972.
\bibitem{GS97}
  H.W. Gould, H.M. Srivastava,
  Some Combinatorial Identities Associated with the Vandermonde Convolution,
  \textit{Appl. Math. Comput.} 84~(1997)~97-102. 
\bibitem{GKP94}
  R.L. Graham, D.E. Knuth, O. Patashnik,
  \textit{Concrete Mathematics, A Foundation for Computer Science}, 2nd ed.,
  Addison-Wesley, 1994.
\bibitem{Knopp}
  K. Knopp,
  \textit{Theory and Application of Infinite Series},
  Dover Publications, 1990.
\bibitem{K08}
  S.G. Krantz,
  \textit{A Guide to Complex Variables},
  The Mathematical Association of America, 2008.
\bibitem{K18}
  M.J. Kronenburg,
  A Generalization of the Chu-Vandermonde Convolution and some Harmonic Number Identities,
  \href{https://arxiv.org/abs/1701.02768}{{\tt arXiv:1701.02768}}{\tt~[math.CO]}
\bibitem{K19}
  M.J. Kronenburg,
  Some Weighted Generalized Fibonacci Number Summation Identities, Part 1,
  \href{https://arxiv.org/abs/1903.01407}{{\tt arXiv:1903.01407}}{\tt~[math.NT]}
\bibitem{K15}
  M.J. Kronenburg,
  The Binomial Coefficient for Negative Arguments,
  \href{http://arxiv.org/abs/1105.3689}{{\tt arXiv:1105.3689}}{\tt~[math.CO]}
\bibitem{M18}
  R. Me\v{s}trovi\'{c},
  Several Generalizations and Variations of Chu-Vandermonde Identity,
  \href{https://arxiv.org/abs/1807.10604}{{\tt arXiv:1807.10604}}{\tt~[math.CO]}
\bibitem{RW00}
T. Rowland, E.W. Weisstein, \textit{Complex Residue}.
From Mathworld - A Wolfram Web Resource.
\href{https://mathworld.wolfram.com/ComplexResidue.html}
{{\tt https://mathworld.wolfram.com/ComplexResidue.html}}
\end{thebibliography}
\end{document}